\documentclass[preprint,1p]{elsarticle}

\makeatletter
 \def\ps@pprintTitle{%
 	\let\@oddhead\@empty
 	\let\@evenhead\@empty
 	\def\@oddfoot{\footnotesize\itshape
 		{} \hfill}%
 	\let\@evenfoot\@oddfoot
 }
\makeatother
\usepackage[unicode]{hyperref}
\usepackage[makeroom]{cancel}
 
\usepackage{soul}
\usepackage{tikz}

\usepackage{latexsym}
\usepackage{indentfirst}
\usepackage{amsxtra}
\usepackage{amssymb}
\usepackage{amsthm}
\usepackage{amsmath}

\usepackage{MnSymbol}
\usepackage{mathrsfs} 
\usepackage[scr=rsfs,cal=boondox]{mathalfa}
\usepackage{xcolor}
\usepackage{color}
\usepackage{amsfonts}

\usepackage[capitalise]{cleveref}
\usepackage{tikz-cd}
\usepackage{amscd}
\usepackage{pst-node}
\DeclareMathOperator{\Map}{Map}

\usepackage[capitalise]{cleveref}

\newtheorem{theor}{Theorem}[section]
\newtheorem*{theor*}{Theorem}
\newtheorem{prop}[theor]{Proposition}
\newtheorem{lemma}[theor]{Lemma}
\newtheorem{cor}[theor]{Corollary}
\newtheorem*{cor*}{Corollary}
\theoremstyle{definition}               %stile roman 
\newtheorem{defin}[theor]{Definition}
\newtheorem{ex}{Example}

\newtheorem{rem}[theor]{Remark}

\newtheorem*{que*}{Question}

\newtheorem*{conv*}{Convention}

\DeclareMathOperator{\Sym}{Sym}

\DeclareMathOperator{\End}{End}
\DeclareMathOperator{\id}{id}
\DeclareMathOperator{\Ann}{Ann}

\DeclareMathOperator{\im}{Im}

\DeclareMathOperator{\Soc}{Soc}
\DeclareMathOperator{\Fix}{Fix}

\DeclareMathOperator{\E}{E}
\DeclareMathOperator{\lcm}{lcm}

\newcommand{\phii}[2]{\phi_{#1,#2}}

\begin{document}

\begin{frontmatter}

\title{Solutions of the Yang-Baxter equation\\ and strong semilattices of skew braces}
	\tnotetext[mytitlenote]{This work was partially supported by the Dipartimento di Matematica e Fisica ``Ennio De Giorgi'' - Università del Salento. The authors are members of GNSAGA (INdAM) and the non-profit association ADV-AGTA.	
	}

\author{Francesco CATINO}
\ead{francesco.catino@unisalento.it}

\author{Marzia MAZZOTTA%\corref{cor2}
}
\ead{marzia.mazzotta@unisalento.it}

\author{Paola STEFANELLI%\corref{cor3}
}
\ead{paola.stefanelli@unisalento.it}


\address{Dipartimento di Matematica e Fisica ``Ennio De Giorgi", Universit\`a del Salento, \\ Via Provinciale Lecce-Arnesano, 73100 Lecce (Italy)}

\begin{abstract}
We prove that any set-theoretic solution of the Yang-Baxter equation associated to a dual weak brace is a strong semilattice of non-degenerate bijective solutions. 
This fact makes use of the description of any dual weak brace $S$  we provide in terms of strong semilattice $Y$ of skew braces $B_\alpha$, with $\alpha \in Y$.
Additionally, we describe the ideals of $S$ and study its nilpotency by correlating it to that of each skew brace $B_\alpha$.
\end{abstract}

\begin{keyword}
 Quantum Yang--Baxter equation \sep set-theoretic solution \sep Clifford semigroup \sep  weak brace \sep skew brace \sep brace
 \MSC[2020]  16T25\sep 81R50 \sep 16Y99 \sep 20M18
\end{keyword}
% ----------------------
\end{frontmatter}

%**********************************
\section*{Introduction}
%**********************************

The quantum Yang–Baxter equation takes its name from two independent works by C.N. Yang \cite{Ya67} and R.J. Baxter \cite{Ba72}. It is an important tool in several fields of research, among these are statistical
mechanics, quantum field theory, and quantum group theory, whose study of solutions has been a major research area for the past 50 years. 
The challenging problem of determining all the set-theoretic solutions arose in 1992 in the paper by Drinfel'd \cite{Dr92} and is still open. Into the specific, given a set $S$, a map
$r:S\times S\to S\times S$ is said to be a \emph{set-theoretic solution of the Yang-Baxter equation}, or briefly a \emph{solution}, if $r$ satisfies the braid identity
\begin{align*}
\left(r\times\id_S\right)
\left(\id_S\times r\right)
\left(r\times\id_S\right)
= 
\left(\id_S\times r\right)
\left(r\times\id_S\right)
\left(\id_S\times r\right).
\end{align*}
Writing $r\left(x,y\right) = \left(\lambda_x\left(y\right), \rho_y\left(x\right)\right)$, with $\lambda_x, \rho_x$ maps from $S$ into itself,  then $r$ is \emph{left non-degenerate} if $\lambda_x\in \Sym_S$, \emph{right non-degenerate} if $\rho_x\in \Sym_S$,  for every $x\in S$, and \emph{non-degenerate} if $r$ is both left and right non-degenerate.\\
Several techniques for constructing solutions starting from known solutions have been introduced over the years. For the purposes of this paper, we mention the \emph{strong semilattice of solutions} \cite{CCoSt20x-2}, which is a method that allows for determining solutions starting from a semilattice $Y$, a family of disjoint sets $\left\{X_{\alpha}\ |\ \alpha \in Y\right\}$, and solutions $r_\alpha$ defined on these sets. \\
In 2007,  Rump \cite{Ru07} innovatively showed how an involutive non-degenerate solution can be obtained starting from the special algebraic structure of \emph{brace}.  This type of approach 
had a large following in the last few years and other similar structures have been studied. Among these, we mention the \emph{weak brace} \cite{CaMaMiSt22}
that is a triple $\left(S,+,\circ\right)$ such that $\left(S,+\right)$ and $\left(S,\circ\right)$ are inverse semigroups satisfying
\begin{align*}
   a\circ\left(b + c\right)
    = a\circ b - a + a\circ c 
    \qquad \text{and} \qquad 
    a\circ a^- = -a + a,
\end{align*}
for all $a,b,c\in S$, where $-a$ and $a^-$ denote the inverses of $a$ with respect to $+$ and $\circ$, respectively. Clearly, the sets of the idempotents $\E(S,+)$ and $\E(S, \circ)$ coincide, so we will simply write $\E(S)$. In particular, if $|\E(S)|=1$, then $(S, +)$ and $(S, \circ)$ are groups having the same identity, and so $(S, +, \circ)$ is a \emph{skew brace} \cite{GuVe17} which is a brace if the group $(S,+ )$ is abelian.
Necessarily, the additive structure  is a Clifford semigroup, instead, in general, the multiplicative one is not. A class of weak braces having $(S, \circ)$ as a Clifford semigroup is obtained in \cite[Theorem 16]{CaMaSt23}. We call \emph{dual} the weak braces for which $(S,\circ)$ is Clifford.\\
Any weak brace $(S, +, \circ)$ gives rise to a solution $r:S\times S\to S\times S$ defined by
$$r\left(a,b\right)
    = \left(-a + a\circ b, \ \left(-a + a\circ b\right)^-\circ a\circ b\right),$$ 
for all $a,b\in S$ (see \cite[Theorem 11]{CaMaMiSt22}). 
Such a map $r$ is close to being bijective, and, in the case of a dual one, $r$ is also close to being non-degenerate (see \cite[pp. 604-605]{CaMaSt23}). Besides, a new family of solutions coming from dual weak braces has been investigated in \cite{MaRySt23x}.
For this kind of structure, a notion of \emph{ideal} has been also introduced in \cite[Definition 20]{CaMaSt23}. Moreover, it turns out that the quotient structure is a new dual weak brace with semilattice of idempotents isomorphic to $\E\left(S\right)$.

In this paper, we entirely describe the structure of a dual weak brace $(S, +, \circ)$,  by showing that it is a strong semilattice $Y$ of skew braces $\left(B_{\alpha}, +, \circ\right)$, for every $\alpha \in Y$, where $\left(B_\alpha, +\right)$ and $\left(B_\alpha, \circ\right)$ are the groups fulfilling the structure of $\left(S, +\right)$ and $\left(S, \circ\right)$, respectively, as Clifford semigroups (see \cite[Theorem 4.2.1]{Ho95}). As a consequence, we prove that the solution $r$ associated to $S$ is the strong semilattice $Y$ of the non-degenerate bijective solutions $r_\alpha$ on each $B_\alpha$.
Any strong semilattice of skew braces is realized by combining skew brace homomorphisms $\phi_{\alpha,\beta}$ from $B_\alpha$ to $B_\beta$ (whenever $\alpha\geq \beta$), thus this further motivates the study of such maps, a problem already emerged in literature (cf. \cite[Problem 10.2]{Ce18}, \cite[Problem 2.18]{Ve19}, \cite{Ze19}, \cite{PuSmZe22}, and \cite{RaYa24}).\\
Despite the obtained description, the skew brace theory is not exhaustive for developing the theory of dual weak braces.  In fact, for instance, although we show the ideals of any dual weak brace $S$ are specific strong semilattices of ideals of every skew brace $B_\alpha$, if we consider known ideals, such as its socle $\Soc(S)$, in general, it is not the strong semilattice $Y$ of each $\Soc\left(B_\alpha\right)$. This led us to deepen the theory of dual weak braces and not just reduce it to the study of every skew brace $B_\alpha$. As a first step, we introduce the binary operation $\cdot$ on $S$ given by $a\cdot b:= -a + a\circ b -b$, for all $a,b\in S$, classically known in the context of radical Jacobson rings. We give some properties that are useful to characterize the ideals of $S$ in terms of the operation $\cdot$. Furthermore, this has allowed us to investigate the right nilpotency and the nilpotency of $S$ by relating them with those of each skew braces $B_\alpha$. We highlight that nilpotency in skew braces has been intensively studied over the years by many authors (see, for instance, \cite{BaEsJiPe23x,BaGu21,BoJed21x,BoFaPo23,CaCeSt22,CeSmVe19,JeVAnVen22,Sm18}) and it is still under investigation above all concerning multipermutation solutions \cite{ESS99}.

\bigskip

%********************************************
\section{Basics on weak braces }
%********************************************

This section aims to give actual results on the structures of weak braces \cite{CaMaMiSt22} paying particular attention to the behavior of the idempotents.
\medskip

To make this paper self-contained and to set up the notation, throughout the paper where it will be needed, we will recall some notions contained in classical books on inverse semigroups, as \cite{ClPr61,Ho95,La98,Pe84}.  
A semigroup $S$ is \emph{inverse} if 
for each $a\in S$, there exists a unique $a^{-1}\in S$ such that $a=aa^{-1}a$ and $a^{-1}=a^{-1}aa^{-1}$. We call such an element $a^{-1}$ the \emph{inverse} of $a$. 
 The class of inverse semigroups is very close to that of groups since  $(a b)^{-1}=b^{-1} a^{-1}$ and $(a^{-1})^{-1}=a$, for all $a,b \in S$. 
Note that $aa^{-1}$ and $a^{-1}a$ are the idempotents of $S$, for every $a \in S$. An inverse semigroup $S$ is called \emph{Clifford} if its idempotents are central, or, equivalently, $aa^{-1}=a^{-1}a$, for every $a \in S$.

\begin{defin} \cite[Definition 5]{CaMaMiSt22}
    Let $S$ be a set endowed with two operations $+$ and $\circ$ such that $\left(S,+\right)$ and $\left(S,\circ\right)$ are inverse semigroups. Then, $(S, +, \circ)$ is a \emph{weak brace} if
 \begin{align*}
    a\circ\left(b + c\right)
    = a\circ b -a + a\circ c\qquad \text{and}\qquad a\circ a^-=-a+a,
 \end{align*}
for all $a,b,c\in S$, where $-a$ and $a^-$ denote the inverses of $a$ with respect to $+$ and $\circ$, respectively.
\end{defin}

\noindent Clearly, the sets of the  idempotents $E(S,+)$ and $E(S,\circ)$ coincide, thus we will simply denote such a set by $E(S)$. Obviously, if $|\E(S)|=1$, then $(S,+, \circ)$ is a skew brace \cite{GuVe17}. \\
In \cite[Theorem 8]{CaMaMiSt22} it is proved that the additive semigroup of any weak brace is necessarily Clifford. In general, the multiplicative one is not (see  \cite[Example 3]{CaMaMiSt22}). Any Clifford semigroup $\left(S, \circ\right)$ determines two trivial weak braces having both Clifford structures,  by setting $a + b:= a\circ b$ or $a + b:= b\circ a$, for all $a,b\in S$. A bigger class of weak braces having $(S, \circ)$ Cifford is studied in \cite[Theorem 16]{CaMaSt23}.

\begin{defin}\cite[Definition 2]{CaMaSt23}
    A weak brace $(S, +, \circ)$ is called \emph{dual} if $(S, \circ)$ is Clifford.
\end{defin}

 \medskip

Given a weak brace $(S,+, \circ)$, let $\lambda: S\to \End(S,+), \,a\mapsto\lambda_a$ and $\rho: S\to \Map(S), \,b\mapsto\rho_b$ be the maps defined by 
\begin{align*}
   \lambda_a\left(b\right) = -a+a\circ b\qquad \text{and} \qquad \rho_b\left(a\right) = \left(-a+a\circ b\right)^{-}\circ a \circ b,
\end{align*}
for all $a,b\in S$, respectively. One has  that $\lambda_a(b)=a\circ \left(a^- +b \right)$, for all $a,b \in S$, and $\lambda_a(E(S)) \subseteq E(S)$. Besides, the map $\lambda$ is a homomorphism of $\left(S,\circ\right)$ into the endomorphism semigroup of $\left(S,+\right)$ and the map $\rho$  is an anti-homomorphism of  $\left(S,\circ\right)$ into the monoid $\Map(S)$. Following \cite[Theorem 11]{CaMaMiSt22}, the map $r:S\times S\to S\times S$ defined by $r\left(a,b\right)
    = \left(\lambda_a(b), \ \rho_b(a)\right)$, 
for all $a,b\in S$, is a solution that has a behaviour close to bijectivity: indeed, considered the solution $r^{op}$ associated to the opposite weak brace
$S^{op}:=(S, +^{op}, \circ)$ of $S$ they hold
  \begin{align*}
      r\, r^{op}\, r = r, \qquad
      r^{op}\, r\, r^{op} = r^{op}, \qquad \text{and}\qquad rr^{op} = r^{op}r.
  \end{align*}
In addition, if $S$ is dual, $r$ has also a behaviour close to the non-degeneracy since
\begin{align*}
    \lambda_a\lambda_{a^-}\lambda_a
    = \lambda_{a},
    \qquad   \lambda_{a^-}\lambda_{a}\lambda_{a^-}
    = \lambda_{a^-}, \quad &\text{and} \qquad
    \lambda_a\lambda_{a^-}
    =\lambda_{a^-}\lambda_{a}, \\  \rho_{a}\rho_{a^-}\rho_{a}
    = \rho_a,
    \qquad
    \rho_{a^-}\rho_{a}\rho_{a^-} = \rho_{a^-},
    \qquad &\text{and} \qquad
    \rho_{a}\rho_{a^-}
    = \rho_{a^-}\rho_{a},
\end{align*}
for every $a \in S$.  Clearly, if $S$ is a skew brace, $r$ is non-degenerate and bijective with  $r^{-1} = r^{op}$ \cite{KoTr20}. Moreover, $S$ is a brace if and only if $r$ is involutive.

In the lemma below, we collect some useful properties of weak braces provided in \cite{CaMaMiSt22,CaMaSt23}. 

\begin{lemma}\label{prop_weak}
Let $(S,+, \circ)$ be a weak brace. Then, the following hold:
\begin{enumerate}
    \item $\lambda_a\left(a^-\right) = - a$,
    \item $a \circ \left(-b\right)=a-a\circ b+a$,
    \item $a\circ b= a + \lambda_a\left(b\right)$,
    \item $a+b=a \circ \lambda_{a^-}\left(b\right),$
\end{enumerate}
for all $a,b \in S$.
\end{lemma}

\medskip

The following key lemma highlights the behavior of idempotents in any weak brace.
\begin{lemma}
\label{pro_dual_idemp}
If $(S,+,\circ)$ is a weak brace and $e\in\E(S)$, then  $\rho_e\left(a\right)=a \circ e$ and
\begin{align}\label{eq:idem+circ}
    \lambda_e(a)=e \circ a=e+a,
\end{align}
for every $a\in S$.\end{lemma}
\begin{proof}
Let $a \in S$. Then, by \cref{prop_weak}-$3.$, $e \circ a=e +\lambda_e\left(a\right)=e-e+e\circ a=\lambda_e\left(a\right)$. 
 As a consequence,  we obtain
$\lambda_e(a)
= e \circ e \circ \left(e+ a\right)
= e \circ \lambda_e(a)
= e + a,
$ where the last equality follows from \cref{prop_weak}-$4.$. In addition, by \cite[Proposition 9-$3.$]{CaMaMiSt22}, $\rho_e\left(a\right)^- = e \circ a^--e=\lambda_e\left(a^-\right)=e \circ a^-$, hence
$\rho_e\left(a\right) = a\circ e$,  thus the claim is satisfied. 
\end{proof}

In light of \cref{pro_dual_idemp}, in any dual weak brace $S$ the set of the idempotents $\E(S)$ gives rise to a structure of trivial weak sub-brace of $S$, the general definition  of which is given below.
\begin{defin}
Let $(S,+, \circ)$ be a weak brace and $H\subseteq S$. Then, $H$ is said to be a \emph{weak sub-brace} of $(S, +, \circ)$ if  $H$ both is an inverse semigroup of $(S,+)$ and $(S,\circ)$ such that $\E(S)\subseteq H$.
\end{defin}

\medskip

\noindent \textbf{Convention.} From now on,  we will only deal with dual weak braces and briefly write $a^0$  to denote the idempotent $a-a=a\circ a^-$, for every $a \in S$.

\medskip

In the last part of this section, we introduce a new operation on a dual weak brace $(S, +, \circ)$
that will allow characterizing its ideals. Such an operation is the usual multiplication in the context of rings and is already known for skew braces \cite{CeSmVe19,KoSmVe21} and cancellative semi-braces  \cite{CaCeSt22}. Specifically, we define the operation $\cdot$ on $S$ given by
  \begin{align*}
      a\cdot b:= -a + a\circ b - b,
  \end{align*}
  for all $a,b\in S$, that can be also rewritten as $a\cdot b = \lambda_a\left(b\right) - b$. As it is usual in brace theory (cf. \cite[Definition 2]{Ru07}), by \eqref{eq:idem+circ}, one has that 
\begin{align}\label{eq:lambda-cdot}
    \lambda_a\left(b\right) =  - a + a\circ b+ b^0
    =  a\cdot b + b,
\end{align}
for all $a,b\in S$. 

\begin{lemma}\label{idemp.puntino}
Let $(S, +, \circ)$ be a dual weak brace and $a,b \in S$. Then, the following are equivalent:
\begin{enumerate}
    \item[(i)] $a \cdot b \in \E(S)$,
    \item[(ii)] $a\cdot b=a^0+b^0$,
    \item[(iii)] $a+b=a \circ b$.
\end{enumerate}\end{lemma}
\begin{proof}
If $a\cdot b\in \E(S)$, then we have that
$$a\cdot b
    = a\cdot b - a\cdot b
    = - a + a\circ b + b^0 - a\circ b +a
    = a^0 + b^0 + \left(a\circ b\right)^0
    \underset{\eqref{eq:idem+circ}}{=} a^0 + b^0.$$
Now, if $a\cdot b = a^0 + b^0$, then
$a + b = a + a^0 + b^0 + b =  a + a\cdot b + b \underset{\eqref{eq:idem+circ}}{=}a\circ b$.
Finally, if $a+b=a \circ b$, we obtain that
$a\cdot b = -a + a\circ b - b = -a + a + b - b = a^0 + b^0\in \E(S)$, which completes the proof.
\end{proof}

\medskip  

Using \cref{idemp.puntino} and \eqref{eq:idem+circ}, we immediately obtain that, for all $a\in S$ and $e\in E\left(S\right)$,
\begin{align}\label{idemp_punt}
    a \cdot e = e \cdot a \in \E(S).
\end{align}

\begin{prop}\label{prop:cdot-circ}
Let $(S, +, \circ)$ be a dual weak brace.  Then, the following are satisfied:
\begin{enumerate}
\item $a \circ b=a+a\cdot b+b$
\item \label{ref_puntino2} $a\cdot \left(b + c\right) = a\cdot b + b + a\cdot c - b$,
\item \label{ref_puntino3} $\left(a\circ b\right)\cdot c = a\cdot\left(b\cdot c\right) + b\cdot c + a\cdot c$,
\end{enumerate}
for all $a,b,c\in S$.\end{prop}
\begin{proof}
Let $a, b, c \in S$. Firstly, by \eqref{eq:idem+circ}, we have
$a \circ b= a^0+a \circ b + b^0=a + a\cdot b + b.
$
Moreover, $a\cdot \left(b + c\right)=-a+a\circ b+b^0-a+a\circ c-c-b=a\cdot b + b + a\cdot c - b.$
Finally, using \eqref{eq:idem+circ}, we get
\begin{align*}
 a&\cdot\left(b\cdot c\right) + b\cdot c + a\cdot c\\
 &=-a+a\circ \left(b \cdot c\right)+a \cdot c \\
 &=-a\circ b+a\circ (b \circ c-c)+a \cdot c &\mbox{by \cref{prop_weak}-$2.$}\\
 &=-a\circ b+a\circ b \circ c-a\circ c+a\circ c-c &\mbox{by \cref{prop_weak}-$2.$}\\
  &=-a \circ b + a \circ b \circ c \circ a^0 \circ c^0-c \\
    &=-a\circ b+a \circ b \circ c-c \\
 &=\left(a \circ b\right)\cdot c,
\end{align*}
which completes our claim.
\end{proof}

\medskip

\begin{cor}
Let $(S, +, \circ)$ be a dual weak brace. Then, they hold:
\begin{enumerate}
    \item $ a\cdot\left(b + e\right)=a\cdot b + a\cdot e    $,
    \item $\left(e + a\right)\cdot b=e\cdot b + a\cdot b$,
\end{enumerate}
for all $a,b \in S$ and $e \in E(S)$.
\end{cor}
\begin{proof}
  If $a,b \in S$ and $e \in E(S)$, we obtain
  \begin{align*}
    a\cdot\left(b + e\right)
    &= a\cdot b + b + a\cdot e -b&\mbox{by \cref{prop:cdot-circ}-$2.$}\\
    &= a\cdot b + b^0 + a\cdot e&\mbox{by \eqref{idemp_punt}}\\
    &= a\cdot b + a\cdot e &\mbox{by \eqref{eq:idem+circ}} 
\end{align*}
and
\begin{align*}
    \left(e + a\right)\cdot b
    &=\left(e \circ a\right) \cdot b
    &\mbox{ by \eqref{eq:idem+circ}}\\
    &=e \cdot (a \cdot b) + a \cdot b+e \cdot b &\mbox{by \cref{prop:cdot-circ}-$3.$}\\
    &=e \circ (a \cdot b)^0 + a \cdot b+e \cdot b&\mbox{ by \eqref{eq:idem+circ}}\\
    &=e \cdot b+a \cdot b &\mbox{by \eqref{idemp_punt}- \eqref{eq:idem+circ}}
\end{align*}
Therefore, the claim follows.
\end{proof}

\bigskip

% --------------------------------
\section{A description of dual weak braces and their solutions}
% ----- 
  
In this section, we provide a description theorem for dual weak braces by showing that they are strong semilattices of specific  skew braces. This description is consistent with the fact that  Clifford semigroups are strong semilattices of groups (see \cite[Theorem II.2.6]{Pe84}).

\smallskip

Let $Y$ be a (lower) semilattice and $\{G_{\alpha}\ | \ \alpha\in Y\}$ a family of disjoint groups. For all $\alpha,\beta\in Y$ such that $\alpha \geq \beta$, let $\phii{\alpha}{\beta}:G_{\alpha}\to G_{\beta}$ be a homomorphism of groups such that
 \begin{enumerate}
     \item $\phii{\alpha}{\alpha}$ is the identical automorphism of $G_{\alpha}$, for every $\alpha \in Y$,
     \item $\phii{\beta}{\gamma}{}\phii{\alpha}{\beta}{} = \phii{\alpha}{\gamma}{}$, for all $\alpha, \beta, \gamma \in Y$ such that $\alpha \geq \beta \geq \gamma$.
 \end{enumerate} 
Then, $S = \displaystyle \bigcup_{\alpha\in Y}\,G_{\alpha}$ endowed with the operation defined by
$
    a\ b:= \phii{\alpha}{\alpha\beta}(a)\ \phii{\beta}{\alpha\beta}(b),
$
for all $a\in G_{\alpha}$ and $b\in G_{\beta}$, is a \emph{Clifford semigroup}, also called  \emph{strong semilattice $Y$ of groups $G_{\alpha}$}, usually written as $S=[Y; G_{\alpha}; \phii{\alpha}{\beta}]$. Conversely, any Clifford semigroup is of this form.

\medskip

\begin{theor}\label{theo_dual_skew}
Let $Y$ be a (lower) semilattice, $\left\{B_{\alpha}\ \left|\ \alpha \in Y\right.\right\}$ a family of disjoint skew braces. For each pair $\alpha,\beta$ of elements of $Y$ such that $\alpha \geq \beta$, let $\phii{\alpha}{\beta}:B_{\alpha}\to B_{\beta}$ be a skew brace homomorphism such that
\begin{enumerate}
    \item $\phii{\alpha}{\alpha}$ is the identical automorphism of $B_{\alpha}$, for every $\alpha \in Y$,
    \item $\phii{\beta}{\gamma}{}\phii{\alpha}{\beta}{} = \phii{\alpha}{\gamma}{}$, for all $\alpha, \beta, \gamma \in Y$ such that $\alpha \geq \beta \geq \gamma$.
\end{enumerate}
Then, $S = \bigcup\left\{B_{\alpha}\ \left|\ \alpha\in Y\right.\right\}$ endowed with the addition and the multiplication, respectively, defined by
\begin{align*}
    a+b:= \phii{\alpha}{\alpha\beta}(a)+\phii{\beta}{\alpha\beta}(b)
    \quad\text{and}\quad
     a\circ b:= \phii{\alpha}{\alpha\beta}(a)\circ\phii{\beta}{\alpha\beta}(b),
\end{align*}
for all $a\in B_{\alpha}$ and $b\in B_{\beta}$, is a dual weak brace with $\E(S)$ isomorphic to $Y$, called \emph{strong semilattice $S$ of skew braces $B_{\alpha}$} and denoted by $S=[Y; B_\alpha;\phii{\alpha}{\beta}]$.  Conversely, any dual weak brace is a strong semilattice of skew braces.\end{theor}
\begin{proof}
    The proof of the sufficient condition is contained \cite[Corollary 6]{CaMaSt23}.
    Conversely, let $\left(S,+,\circ\right)$ be a dual weak brace and let $[Y; B_\alpha;\phii{\alpha}{\beta}]$ and $[Z; H_i;\psi_{{i},{j}}]$ the two  Clifford semigroups $\left(S,+\right)$ and $\left(S,\circ\right)$, respectively. 
    Since $\E\left(S,+\right)$ and $\E\left(S,\circ\right)$ coincide 
    and they are isomorphic to $Y$ and $Z$, respectively, there exists a semillattice isomorphism $f:Y\to Z$, that, formally, can be written as $f\left(\alpha\right) = i_{\alpha}$, for every $\alpha\in Y$. Keeping in mind this fact, it is not restrictive to consider $Y = Z$.
    Now, let $\alpha\in Y$ and show that $B_\alpha = H_\alpha$. Thus, if $a\in B_\alpha$, then there exists $\beta\in Y$ such that $a\in H_\beta$. Denoted by $e_\alpha$ and $e_\beta$ the identities of the groups $\left(B_\alpha, +\right)$ and $\left(H_\beta, \circ\right)$, respectively, note that
    \begin{align*}
        e_\alpha = a - a = a\circ a^- = e_\beta,
    \end{align*}
    hence $a\in B_\alpha\cap H_\alpha\subseteq H_\alpha$. Reversing the role of $+$ and $\circ$, one similarly obtains that $H_\alpha\subseteq B_\alpha$.
    Finally, if $\alpha,\beta\in Y$ and $\alpha\geq \beta$, observe that
    \begin{align*}
      \phi_{\alpha, \beta}\left(a\right)
      = a + e_\beta
      \underset{\eqref{eq:idem+circ}}{=} a\circ e_\beta 
      = \psi_{\alpha,\beta}\left(a\right),
    \end{align*}
for every $a\in B_\alpha$. Therefore, $\left\{B_{\alpha}\ \left|\ \alpha \in Y\right.\right\}$ is a family of disjoint skew braces and the claim follows.
\end{proof}

\smallskip

The following result is a consequence of \cite[Proposition II.2.8]{Pe84} and \cref{theo_dual_skew}.
\begin{prop}
    Two dual weak braces  $S=[Y; B_{\alpha}; \phii{\alpha}{\beta}]$ and $T=[Z; C_{i}; \psi_{i, j}]$ are isomorphic if and only if there exists a semilattice homomorphism $\eta: Y \to Z$ and a family of skew brace isomorphisms $\{\Theta_\alpha: B_\alpha \to C_{\eta(\alpha)}\mid \alpha \in Y\}$ such that
 $\Theta_\beta\phi_{\alpha, \beta}=\psi_{\eta(\alpha), \eta(\beta)}\Theta_\alpha,
$
    for every $\alpha \geq \beta$.
\end{prop}

\medskip

The following example shows that \cref{theo_dual_skew} allows for obtaining not trivial dual weak braces even if we start from trivial skew braces.
\begin{ex}\label{ex_C3_sym3}
   Let us consider $Y=\{\alpha, \beta\}$, with $\alpha > \beta$, $B_\alpha$ and $B_\beta$ the trivial skew braces on the cyclic group $C_3$ and on the symmetric group $\Sym_3$ of order $3$, respectively, and $\phi_{\alpha, \beta}: C_3 \to \Sym_3$ the homomorphism given by $\phi_{\alpha, \beta}(0)=\id_3$, $\phi_{\alpha, \beta}(1)=(123)$, $\phi_{\alpha, \beta}(2)=(132)$. Then, $S=[Y; B_\alpha;\phii{\alpha}{\beta}]$ is a non trivial dual weak brace.
\end{ex}

\medskip

The following natural question arises intending to concretely construct dual weak braces.

\medskip
\noindent \textbf{Question} \emph{Determining all skew brace homomorphisms.}

\medskip
\noindent Note that the problem of studying homomorphisms between skew braces has emerged yet in the literature, such as \cite[Problem 10.2]{Ce18}, \cite[Problem 2.18]{Ve19}, \cite{Ze19}, \cite{PuSmZe22}, and \cite{RaYa24}, and also in a recent conference talk held by Civino \cite{Ci23AGTA}.

\medskip

Now, by \cref{theo_dual_skew}, we illustrate that the solution $r$ associated to any dual weak brace $S=[Y; B_\alpha;\phi_{\alpha, \beta}]$ is the \emph{strong semilattice of the solutions $r_{\alpha}$} on each skew brace $B_{\alpha}$. In this regard, we recall the more general result contained in \cite[Theorem 4.1]{CCoSt20x-2}.

\begin{theor}\label{th:sol-sss}
	Let $Y$ be a (lower) semilattice, let $\left\{r_{\alpha}\ \left|\ \alpha\in Y \right.\right\}$  be a family of disjoint solutions on $X_{\alpha}$ indexed by $Y$ such that for each pair $\alpha, \beta\in Y$ with $\alpha\geq \beta$, there is a map $\phi_{\alpha, \beta}:X_{\alpha}\to X_{\beta}$. Let $X$ be the union
	$$
	X = \bigcup\left\{X_{\alpha}\ \left|\ \alpha\in Y\right.\right\}
	$$
	and let $r:X\times X \longrightarrow X\times X$ be the map defined as
	\begin{align*}
		r\left(x, y\right):= 
		r_{\alpha\beta}\left(\phi_{\alpha,\alpha\beta}\left(x\right),
		\phi_{\beta, \alpha\beta}\left(y\right) \right),
	\end{align*}	
    for all $x\in X_{\alpha}$ and $y\in X_{\beta}$. Then, if the following conditions are satisfied
	\begin{enumerate}
		\item $\phi_{\alpha, \alpha}$ is the identity map of $X_{\alpha}$, for every $\alpha\in Y$,
		\item $\phi_{\beta, \gamma}\phi_{\alpha, \beta} = \phi_{\alpha, \gamma}$, for all $\alpha, \beta, \gamma \in Y$ such that $\alpha \geq \beta \geq \gamma$,
		\item 
		$\left(\phi_{\alpha,\beta}\times \phi_{\alpha,\beta}\right)r_{\alpha}
		= r_{\beta}\left(\phi_{\alpha, \beta}\times \phi_{\alpha, \beta}\right)$, for all $\alpha, \beta\in Y$ such that  $\alpha \geq \beta$.	
	\end{enumerate}
then $r$ is a solution on $X$ called a \emph{strong semilattice of solutions $r_{\alpha}$ indexed by $Y$}.
\end{theor}	
	
	\medskip
	\noindent
Given a dual weak brace $S=[Y; B_\alpha;\phii{\alpha}{\beta}]$, clearly the conditions of \cref{th:sol-sss} are satisfied by each solution $r_{\alpha}$ on each skew brace $B_\alpha$. In particular, the condition $3.$ follows by the fact that the map $\phi_{\alpha,\beta}: B_\alpha \to B_\beta$ is a skew brace homomorphism, for every $\alpha\geq \beta$. Moreover, it is easy to see that also the converse is true. Thus, we can sum it all up in the following result.
	\begin{prop}
Let $S=[Y; B_\alpha;\phii{\alpha}{\beta}]$ be a dual weak brace.	Then, the solution $r$ associated to $S$ is the strong semilattice of the solutions $r_\alpha$ associated to each skew brace $B_{\alpha}$. 
\end{prop}

\medskip

	To be thorough, we know from \cite[Theorem 4.13]{SmVe18} that in the finite case the order of each solution $r_{\alpha}$, or equivalently the period $p\left(r_{\alpha}\right)$ of $r_\alpha$, for every $\alpha \in Y$,  is even and, by  \cite[Theorem 5.2]{CCoSt20x-2}, we can establish the order of the overall solution $r$. In the next result, for a finite dual weak brace $S=[Y; B_\alpha;\phii{\alpha}{\beta}]$ we mean that $Y$ and $B_\alpha$ are finite, for every $\alpha \in Y$. 

\begin{cor}\label{cor_period_index}
Let $S=[Y; B_\alpha;\phii{\alpha}{\beta}]$ be a finite dual weak brace and $r$ the solution associated to $S$. Then, $r^{2k+1}=r$ with $2k=\lcm\{p\left(r_{\alpha}\right) \, \mid \, \alpha \in Y\}$, where $r_\alpha$ is the solution on each skew brace $B_\alpha$, for every $\alpha \in Y$.\end{cor}
\smallskip

\begin{cor} 
Let  $S=[Y; B_\alpha;\phii{\alpha}{\beta}]$ be a strong semilattice of braces such that $Y$ is finite. Then, the solution $r$ associated to $S$ is cubic, namely, $r^3=r$.
\end{cor}

\bigskip

% --------------------------------
\section{Ideals of dual weak braces}
% -------------------------------
In this section, we entirely describe the structure of any ideal of a dual weak brace $S=[Y; B_\alpha; \phi_{\alpha, \beta}]$ by means of ideals of each skew brace $B_\alpha$. 
\medskip

Let us start by recalling the notion of ideal contained in \cite[Definition 20]{CaMaSt23}. To this end, by a \emph{normal subsemigroup} $I$ of a Clifford semigroup $S$ we will mean a subset $I$ of $S$ such that
\begin{enumerate}
    \item $\E(S) \subseteq I$ (i.e., $I$ is \emph{full} in $S$),
    \item $\forall a, b \in I$ \quad $ab \in I$ \quad \text{and} \quad $a^{-1}\in I$,
    \item $\forall a \in S$ \quad $a^{-1}Ia\subseteq I$.
    \end{enumerate}
Note that such a  definition is equivalent to that of \emph{normal subset} contained in \cite[Definition VI.1.2]{PeRe99}.
% If $\xi$ is a congruence on $\E(S)$, the pair $(\xi, I)$ is named a \emph{congruence pair} for $S$ if $I$ is a normal subsemigroup of $S$ and
% \begin{align*}
%   \forall\ a \in S,\ e\in \E(S)  \quad ae\in I \ \  \text{and} \ \ e\ \xi \ a^{0} \ \Longrightarrow \ a\in I.
% \end{align*}

\begin{defin}(cf. \cite[Definition 20]{CaMaSt23})\label{def:ideal-dwb}
Let $(S,+,\circ)$ be a dual weak brace and $I$ a subset of $S$. Then, $I$ is an \emph{ideal} of $(S,+,\circ)$ if they hold
\begin{enumerate}
    \item $I$ is a normal subsemigroup of $(S,+)$,
    \item $\lambda_a(I)\subseteq I$, for every $a\in S$,
    \item $I$ is a normal subsemigroup of $(S,\circ)$.
\end{enumerate}
\end{defin}
\noindent It follows from the definition of every ideal $I$ that $\E(S) \subseteq I$. In particular, $I$ is a dual weak sub-brace of $S$, thus $S$ and $\E(S)$ are trivial ideals of $S$.  According to \cite[Proposition 23]{CaMaSt23}, another known ideal of $S$ is its \emph{socle} (or \emph{left annihilator}), i.e., the set
\begin{align*}
   \Soc\left(S\right)=\{ a  \, | \,a \in S, \, \,\forall \ b \in S \quad a+b=a \circ b \quad \text{and} \quad a+b=b+a \}.
\end{align*} 
In addition, note that if $I$ and $J$ are ideals of $S$, then, as proved in \cite[Proposition 24]{CaMaSt23}, $I+J$ and $I \circ J$ also are, and one can easily see from \cref{prop_weak} that they coincide.

\medskip

In light of \cref{theo_dual_skew}, we can give the following structure theorem for every ideal of a dual weak brace.

\begin{theor}\label{teo_descr_ideali}
    Let $S=[Y; B_{\alpha}; \phii{\alpha}{\beta}]$ be a dual weak brace, $I_\alpha$ an ideal of each skew brace $B_\alpha$, and set $\psi_{\alpha,\beta} := {\phi_{\alpha,\beta}}_{|_{I_\alpha}}$, for all $\alpha \geq \beta$. If $\phii{\alpha}{\beta}(I_\alpha)\subseteq I_\beta$, for any $\alpha > \beta$, then $I=[Y; I_{\alpha}; \psi_{\alpha, \beta}]$ is an ideal of $S$ and, conversely, every ideal of $S$ is of this form.\end{theor}
\begin{proof}
To show that $I=[Y; I_{\alpha}; \psi_{\alpha, \beta}]$ is an ideal of $S$, by \cite[Exercises III.1.9(ii)]{Pe84}, it is enough to prove the $\lambda$-invariance of $I$. If $a\in S$ and $x\in I$, there exist $\alpha,\gamma\in Y$ such that $a\in B_{\alpha}$ and $x\in I_{\gamma}$, and so, since by the assumption $ \phi_{\gamma, \alpha\gamma}\left(x\right)\in I_{\alpha\gamma}$, we have that
\begin{align*}
    \lambda_a\left(x\right) 
    = -a + \phi_{\alpha, \alpha\gamma}\left(a\right)\circ \phi_{\gamma, \alpha\gamma}\left(x\right)
    = \lambda_{\phi_{\alpha,\alpha\gamma}\left(a\right)}\phi_{\gamma,\alpha\gamma}\left(x\right)
    \in I_{\alpha\gamma},
\end{align*}
hence $\lambda_a\left(x\right)\in I$.\\
Vice versa, if $I$ is an ideal of $(S,+, \circ)$, since $I$ is a dual weak brace too, it follows by \cref{theo_dual_skew} and  \cite[Exercises III.1.9(ii)]{Pe84} that  $I=[Y; I_{\alpha}; \psi_{\alpha, \beta}]$, with $I_\alpha=B_\alpha \cap I$, for every $\alpha \in Y$, $\psi_{\alpha,\beta} := {\phi_{\alpha,\beta}}_{|_{I_\alpha}}$, for all $\alpha \geq \beta$, such that $\phii{\alpha}{\beta}(I_\alpha)\subseteq I_\beta$, for any $\alpha > \beta$. To get the claim, we prove the $\lambda$-invariance of each $I_\alpha$, i.e., $I_\alpha$ is an ideal of the skew brace $B_\alpha$, for every $\alpha \in Y$. Indeed, we have that
   $ \lambda_a(x) \in B_\alpha \cap I=I_\alpha$,
  for all $a \in B_\alpha$ and $x \in I_\alpha$. Therefore, the claim follows.
 \end{proof}

\medskip

By \cref{teo_descr_ideali}, the ideals of any dual weak brace $S=[Y; B_{\alpha}; \phii{\alpha}{\beta}]$ are strong semilattices $Y$ of certain ideals $I_\alpha$ of each skew brace $B_\alpha$. However, the skew brace theory is not exhaustive for developing the theory of dual weak braces. In fact, for instance, if we consider the ideal $\Soc(S)$, in general, it is not the strong semilattice $Y$ of the socles $\Soc\left(B_\alpha\right)$ of $B_{\alpha}$.  Indeed, one clearly has that $\displaystyle{\Soc(S) \subseteq \bigcupdot_{\alpha \in Y} \ \Soc\left(B_\alpha\right)}$, but the other inclusion is not true in general. The dual weak brace in \cref{ex_C3_sym3} is such an example of this fact. In the next result, we show when they coincide.

\begin{prop}\label{soc-ideale}
  Let  $S=[Y; B_{\alpha}; \phii{\alpha}{\beta}]$ be a dual weak brace, $\psi_{\alpha,\beta} := {\phi_{\alpha,\beta}}_{|_{\Soc(B_\alpha)}}$, for all $\alpha \geq \beta$, and assume that $I:=[Y; \Soc\left(B_{\alpha}\right); \psi_{\alpha, \beta}]$ is an ideal of $S$. Then, $I=\Soc(S)$.\end{prop}
  \begin{proof}
     At first, by \cref{teo_descr_ideali},  $\phii{\alpha}{\beta}(\Soc(B_\alpha))\subseteq \Soc(B_\beta)$, for any $\alpha > \beta$. Let $\alpha \in Y$ and $a \in \Soc(B_\alpha)$. Then, if $b \in S$, there exists $\beta \in Y$ such that $b \in B_\beta$, thus we obtain
     \begin{align*}
         a+b =
         \underbrace{\phii{\alpha}{\alpha\beta}(a)}_{\in\, \Soc(B_{\alpha\beta})}\ + \ \phii{\beta}{\alpha\beta}(b)
         = \phii{\alpha}{\alpha\beta}(a)\circ \phii{\beta}{\alpha\beta}(b)
         = a\circ b.
     \end{align*}
     Similarly, we get $a + b = b + a$, hence $a \in\Soc(S)$. Therefore, we get the claim.
  \end{proof}

\medskip

According to \cite[Theorem 21]{CaMaSt23}, ideals allow for obtaining quotient structures.  Into the specific, if $I$ is an ideal of a dual weak brace $(S,+,\circ)$, then the relation $\sim_{I}$ on $S$ given by
\begin{align*}
    \forall \ a,b\in S\qquad a\sim_I b  
    \ \iff \
    a^0 = b^0
     \ \  \text{and} \ \ -a + b\in I
\end{align*}
is an idempotent separating congruence of $(S,+, \circ)$.  Furthermore, $S/I:=S/\sim_I$ is a dual weak brace with semilattice of idempotents isomorphic to $\E\left(S\right)$.\\
As it is natural to expect, we can consider the canonical epimorphism $\chi: S\to S/I, \ a\mapsto a+I$ and give the usual homomorphism theorems.

\begin{theor}\label{teo_omo}
Let $\phi: S\rightarrow T$ be a homomorphism between two dual weak braces $(S,+,\circ)$ and $(T,+,\circ)$. Then, the following statements hold:
\begin{enumerate}
    \item $\im\,\phi$ is a dual weak sub-brace of $(T,+,\circ)$;
    \item $\ker\phi=\{a  \mid a \in S \quad \exists\, e \in \E(S) \quad \phi(a)=\phi(e)  \}$ is an ideal of $(S, +, \circ)$;
    \item there exists a monomorphism $\widetilde{\phi}:S/{\ker\phi}\rightarrow T$ of dual weak braces such that $\im\,\widetilde{\phi} = \im\,\phi$ and the diagram
    \begin{center}
  \begin{tikzcd}[column sep=50pt, row sep=3.5em]
S  \arrow[r, "\phi"] \arrow[d," \chi " swap ]& T  \\
S/{\ker\phi} \arrow[ur, "\widetilde{\phi}" swap] &
\end{tikzcd}      
    \end{center}
    commutes.
\end{enumerate}
\end{theor}

\medskip

\begin{cor}
Let  $(S,+,\circ)$ be a dual weak brace, $I$ an ideal of $S$, and $H$ a dual weak sub-brace of $S$.  Then,
\begin{enumerate}
    \item $I+H$ is a dual weak sub-brace of $S$;
    \item $I$ is an ideal of $I+H$;
    \item $I \cap H$ is an ideal of $H$;
    \item $H/ {\left(I\cap H\right)}$ is isomorphic to $\left(I+H\right)/ I$.
\end{enumerate}
\end{cor}

\medskip

\begin{cor}
 Let $(S, +, \circ)$ be a dual weak brace and $I, J$ ideals of $S$ such that $J \subset I$. Then, $S/I$ is isomorphic to $\left(S/J\right)/\left(I/J\right)$.
\end{cor}

\medskip

\begin{cor}
  Let $I$ be an ideal of a dual weak brace $(S,+, \circ)$. There is a one-to-one correspondence between the set of ideals of $S$ containing $I$ and the set of ideals of $S/I$. Moreover, ideals of $S$ containing $I$ correspond to ideals of $S/I$.
\end{cor}

\bigskip

%****************************

\section{Left ideals of dual weak braces}
%************ ****************
In this section, we provide a characterization of ideals of any dual weak brace which makes use of the concept of left ideals and of the operation $\cdot$.

\medskip

Let us start by introducing the notions of the left ideal and strong left ideal of a dual weak brace, consistently with \cite{CeSmVe19} and \cite[Definition 2.3]{JeKuVaVe19}.
\begin{defin}
Let $(S,+,\circ)$ be a dual weak brace. Then, a subset $I$ of $S$ is a \emph{left ideal} of $(S,+,\circ)$ if 
\begin{enumerate}
\item $I$ is a full inverse subsemigroup of $(S,+)$,
    \item $\lambda_a(I)\subseteq I$, for every $a\in S$.
\end{enumerate}
A left ideal $I$ is a \emph{strong left ideal} if $I$ is a normal subsemigroup of $(S,+)$.
\end{defin}
\noindent  Note that any left ideal $I$ of a dual weak brace $S$ is a full inverse subsemigroup of $\left(S,\circ\right)$. Indeed, by \cref{prop_weak}, if $a,b\in I$, then
$a\circ b = a + \lambda_a\left(b\right)\in I$ and  $a^- = \lambda_{a^-}\left(-a\right)\in I$. 

\medskip

Similarly to \cref{teo_descr_ideali}, we can describe the structure of any strong left ideal, as follows.
\begin{prop}
     Let $S=[Y; B_{\alpha}; \phii{\alpha}{\beta}]$ be a dual weak brace, $I_\alpha$ a strong left ideal of each skew brace $B_\alpha$, and set $\psi_{\alpha,\beta} := {\phi_{\alpha,\beta}}_{|_{I_\alpha}}$, for all $\alpha \geq \beta$. If $\phii{\alpha}{\beta}(I_\alpha)\subseteq I_\beta$, for any $\alpha > \beta$, then $I=[Y; I_{\alpha}; \psi_{\alpha, \beta}]$ is a strong left ideal of $S$ and, conversely, every strong left ideal of $S$ is of this form.
\end{prop}

\smallskip

\begin{ex}
Let $(S, +, \circ)$ be a dual weak brace. Every full endomorphism-invariant subsemigroup $I$ of $(S,+)$, i.e., $\varphi(I) \subseteq I$, for every $\varphi \in \End(S, +)$, is a strong left ideal (cf. \cite[cf. Example 2.4]{JeKuVaVe19}).
\end{ex}
\smallskip

In skew braces theory \cite{CeSmVe19}, a known left ideal of a skew brace $(B, +, \circ)$ is $\Fix(B)$, that is the set of the elements of $B$ that are fixed by the map $\lambda_a$, for every $a \in B$. For a dual weak brace $S$, it can be defined in a similar way, as we show below.
\begin{prop}
Let $(S,+, \circ)$ be a dual weak brace. Then, the following set
\begin{align*}
    \Fix(S):&=\{b  \ | \ b \in S, \, \,\forall \ a \in S \quad a+b=a \circ b\}\\
    &=\{b  \ | \ b \in S, \, \,\forall \ a \in S \quad \lambda_a(b)=a^0+b\}
\end{align*}
is a left ideal of $S$.\end{prop}
\begin{proof}
Initially, by \eqref{eq:idem+circ}, $E(S) \subseteq \Fix(S)$. Moreover, $(\Fix(S), +)$ is trivially closed with respect to $+$ and, by \cref{prop_weak}-$2.$, $
    a \circ (-b)=a-a\circ b+a=a-b+a^0=a-b$,
for all $a \in S$ and $b \in \Fix(S)$. Finally, if $a,c \in S$ and $b \in \Fix(S)$, it follows that 
\begin{align*}
  c +\lambda_a(b)&=c+c^0+a^0+b &\mbox{$b\in \Fix(S)$}\\
  &=c+c^0\circ a^0+b
  &\mbox{by \eqref{eq:idem+circ}}\\
  &=c-c\circ a+c\circ a+b\\
  &=c-c\circ a+c\circ a\circ b &\mbox{$b\in \Fix(S)$}\\
  &=c+\lambda_{c \circ a}(b)\\
  &=c \circ \lambda_a(b) &\mbox{by \cref{prop_weak}-$3.$}
\end{align*}
Therefore, $\Fix(S)$ is a left ideal of $S$.
\end{proof}

\medskip

\begin{defin}
If $(S, +, \circ)$ is a dual weak brace, then  we call the set $Zl(S):=\Fix(S) \cap \zeta(S,+)$ the \emph{left center} (or \emph{right annihilator}) of $S$. 
\end{defin}
\noindent It is easy to check that $Zl(S)$ is a strong left ideal of any dual weak brace $S$. In general, $Zl(S)$ is not an ideal as shown in the context of skew braces in \cite[Example 5.7]{BaGu21}.

\medskip
In the next, if $X$ and $Y$ are subsets of a dual weak brace $(S, +, \circ)$, we denote by $X \cdot Y$ the additive inverse subsemigroup of $S$ generated by the elements of the form $x \cdot y$, with $x \in X$ and $y \in Y$ (cf. \cite[Definition II.1.11]{Pe84}).

\begin{prop}\label{lemma-ideali}
Let $(S,+, \circ)$ be a dual weak brace. 
Then, the following hold:
\begin{enumerate}
    \item  a full inverse subsemigroup $I$ of $\left(S,+\right)$ is a left ideal of $S$ if and only if $S \cdot I \subseteq I$;
    \item if $I$ is an ideal of $S$, then  $S\cdot I\subseteq I$ and $I\cdot S\subseteq I$;
    \item a normal subsemigroup $I$ of $\left(S,+\right)$ is an ideal of $S$ if and only if $I\cdot S\subseteq I$ and $\lambda_a\left(I\right)\subseteq I$, for every $a\in S$.
\end{enumerate}\end{prop}
\begin{proof}
$1.$ \ If $I$ is a left ideal of $S$, $a \in S$ and $x \in I$, then  $a\cdot x=\lambda_a\left(x\right)-x \in I$. Conversely, if $S\cdot I\subseteq I$, $a\in S$ and $x\in I$, by \eqref{eq:lambda-cdot}, we obtain that
$\lambda_a\left(x\right) =  a\cdot x + x\in I$.\\
$2.$ \ By $1.$, it is enough to see that if $a \in S$ and $x \in I$, then
\begin{align*}
    x \cdot a=-x+a-a+a \circ a^-\circ x\circ a-a=-x+\underbrace{a+\lambda_a\left(a^-\circ x \circ a\right)-a}_{\in I} \in I. 
\end{align*} 
$3.$ \ It is enough to show that if $a\in S$ and $x\in I$, then by \cref{prop_weak}
\begin{align*}
    a\circ x\circ a^- 
    &= a + \lambda_a\left(x + \lambda_x\left(a^-\right)\right)
    =a + \lambda_a\left(x \right) + \lambda_{a\circ x}\left(a^-\right) + a - a\\
    &= a + \lambda_a\left(x + \lambda_x\left(a^-\right) - a^-\right) - a
    = a + \lambda_a\underbrace{\left(x +x\cdot a^-\right)}_{\in I} - a\in I,
\end{align*}
which completes the proof.
\end{proof}

\medskip

As a consequence of \cref{lemma-ideali}, one can characterize the ideals of a dual weak brace as is usual in ring theory.
\begin{cor}\label{th:prop-divorante}
Let $(S, +, \circ)$ be a dual weak brace. Then, a normal subsemigroup $I$ of $\left(S,+\right)$ is an ideal of $S$ if and only if $S\cdot I\subseteq I$ and $I\cdot S\subseteq I$.
\end{cor}

\medskip

To conclude this section, similarly to \cite[Proposition 1.3]{BoJed21x}, we show how to obtain instances of left ideals starting from ideals. If $X,Y$ are ideals of a dual weak brace $S$, we denote by
    $\left[X,\, Y\right]_+ 
        = \bigl\langle\, \left[x,\, y\right]_+  \ | \ x\in X,\ y\in Y\, \bigr\rangle$.
\begin{prop}
Let $\left(S,+,\circ\right)$ be a dual weak brace and $I,J$ ideals of $S$. Then, $I\cdot J$,  $I\cdot J + J\cdot I$, and $\left[I,\, J\right]_+$ are left ideals of $S$.\end{prop}
\begin{proof}
Initially, note that $I\cdot J$ and $I\cdot J + J\cdot I$ trivially contain $\E\left(S\right)$. Moreover, $I\cdot J$ clearly is an inverse subsemigroups of $\left(S, +\right)$ and if $s\in S$, $x\in I$, and $h\in J$, then, by \cref{prop_weak}-$2.$ and \eqref{idemp_punt},
\begin{align*}
    s\cdot\left(x\cdot h\right)
    &= \lambda_{s\circ x}\left(h\right) - \lambda_{s}\left(h\right) - x\cdot h
    = \lambda_{s\circ x\circ s^-}\lambda_{s}\left(h\right) - \lambda_{s}\left(h\right) - x\cdot h\\
    &= \left(s\circ x\circ s^-\right)\cdot \lambda_{s}\left(h\right) - x\cdot h\in I\cdot J,
\end{align*}
hence, by \cref{lemma-ideali}-$1.$, $I\cdot J$ is a left ideal of $S$.\\
To prove that $I\cdot J + J\cdot I$ is an inverse subsemigroup of $\left(S,+\right)$, let us check that 
$J\cdot I + I\cdot J\subseteq I\cdot J + J\cdot I$. If $x,y\in I$ and $h,k\in J$, since
$\left(y\cdot\left(h\cdot x\right)\right)^0 =
y^0 + y^0\circ\left(h\cdot x\right)^0 + \left(h\cdot x\right)^0
= \left(h\cdot x\right)^0 + y^0$ and
$y\cdot k = y^0 + y\cdot k$, 
it follows that
\begin{align*}
    h\cdot x + y\cdot k
    &= \left(y\cdot\left(h\cdot x\right)\right)^0 + h\cdot x + y\cdot k + \left(h\cdot x\right)^0\\  
    &= - y\cdot\left(h\cdot x\right)
    + y\cdot\left(h\cdot x + k\right) + h\cdot x
    \in I\cdot J + J\cdot I,
\end{align*}
where in the last equality we use \cref{prop:cdot-circ}-$2.$. Furthermore, by the first part of the proof, we trivially get $\lambda_a\left(I\cdot J + J\cdot I\right)\subseteq I\cdot J + J\cdot I$, for every $a\in S$. \\
Finally,  $\left[I,\, J\right]_+$ trivially is a left ideal of $S$. Therefore, the claim is proved.
\end{proof}

\bigskip

% -------------------------------------------------
\section{Right nilpotent dual weak braces}
% -------------------------------------------------

Following \cite{CeSmVe19,Ru07}, if $(S,+,\circ)$ is a dual weak brace, set $S^{(1)}:=S$, we inductively define
    \begin{center}
      $S^{(n+1)}=S^{(n)}\cdot S$,
    \end{center} 
    for every $n\geq 1$, and obtain that it is an ideal, whose proof can be proved similarly to 
\cite[Proposition 2.1]{CeSmVe19}. However, unlike what happens for the socle, here we can use the characterization of ideals in \cref{teo_descr_ideali}.
\begin{prop}
If $(S,+,\circ)$ is a dual weak brace, then $S^{(n)}$ is an ideal of $S$, for every $n \geq 1$.\end{prop}
\begin{proof}
    Initially, by \cref{theo_dual_skew}, we have that $S=[Y; B_\alpha; \phii{\alpha}{\beta}]$ and one clearly obtains $\displaystyle{S^{(n)} \subseteq \bigcupdot_{\alpha \in Y} \ B_\alpha^{(n)}}$. The claim follows by \cref{teo_descr_ideali}, by showing that $\phii{\alpha}{\beta}\left(B_\alpha^{(n)}\right)\subseteq B_\beta^{(n)}$, for any $\alpha > \beta$, by proceeding by an easy induction on $n$. 
\end{proof}

\medskip

\begin{defin}
    A dual weak brace $(S,+,\circ)$ is said to be \emph{right nilpotent} if $S^{(n)}=E(S)$ for some $n\geq 1$.
    The smallest positive integer $m$ such that $S^{(m+1)}=\E(S)$ is called \emph{right nilpotency index}  of $S$.
\end{defin}

\noindent It is easy to check that if $S$ is a right nilpotent dual weak brace, then any ideal $I$ and quotient $S/I$ are right nilpotent as well. At the present state of our knowledge, we do not know whether the vice versa is true in general. Surely, it is true when $I$ is the socle of any dual weak brace, as we show in the following result.
At first, note that $\Soc\left(S\right)$ is right nilpotent  since,  using \cref{idemp.puntino}, it can be rewritten in terms of the $\cdot$ operation, as
\begin{align*}
     \Soc\left(S\right) = \{ a  \, | \,a \in S, \, \,\forall \ b \in S \quad a\cdot b \in \E(S) \quad \text{and} \quad a+b=b+a \}.
\end{align*}

\begin{prop}
Let $(S,+, \circ)$ be a dual weak brace such that $S/\Soc\left(S\right)$ is right nilpotent. Then, $S$ is right nilpotent.\end{prop}
\begin{proof}
Since there exists $m\in\mathbb{N}$ such that $\left(S/\Soc\left(S\right)\right)^{\left(m\right)} = E\left(S/\Soc\left(S\right)\right)$, it follows that $S^{\left(m\right)}\subseteq \Soc\left(S\right)$.
Hence, $S^{\left(m+1\right)} = S^{\left(m\right)}\cdot S\subseteq \Soc\left(S\right)\cdot S\subseteq E\left(S\right)$. 
Thus, the claim is proved.
\end{proof}

\medskip

\begin{defin} \label{def_soc_Series}Let $(S,+, \circ)$ be a dual weak brace. Set $\Soc_0\left(S\right):= E\left(S\right)$, we define $\Soc_n\left(S\right)$ to be the ideal of $S$ containing $\Soc_{n-1}\left(S\right)$ such that
    \begin{align*}
       \Soc_n\left(S\right)/\Soc_{n-1}\left(S\right)=\Soc\left(S/\Soc_{n-1}\left(S\right)\right),
    \end{align*}
    for every positive integer $n$.\end{defin}
\noindent In particular, 
$\Soc_n\left(S\right):=\{a \, | \, a\in S, \, \,  \forall\, b \in S \quad a \cdot b\in \Soc_{n-1}\left(S\right), \,\, [a,b]_{+} \in \Soc_{n-1}\left(S\right)\},$ for every positive integer $n$ (cf. \cite[Remark 28]{CaCeSt22}).
\medskip

\begin{defin}
   Let $(S,+,\circ)$ be a dual weak brace. An \emph{$s$-series} (or \emph{left annihilator series}) of $S$ is a sequence of ideals of $S$
   \begin{align*}
       \E(S) = I_0\subseteq I_1\subseteq I_2\subseteq\cdots\subseteq I_m = S
   \end{align*}
such that $I_{j}/I_{j-1}\subseteq \Soc\left(S/{I_{j-1}}\right)$, for each $j\in \{1,\ldots,m\}$. 
\end{defin}

\medskip
 
\noindent One can check that  $S$ admits an s-series if and only if there exists a positive integer $n$ such that $S=\Soc_n\left(S\right)$, as in the case of the skew braces (cf. \cite[Lemma 2.15]{CeSmVe19}). In light of this fact, we give the following definitions.

\begin{defin}
    A dual weak brace $(S,+,\circ)$ is called \emph{left annihilator nilpotent} if $S$ admits an s-series. Consequently,  if $S$ is left annihilator nilpotent, we call \emph{socle series} (or \emph{upper left annihilator series}) the series introduced in \cref{def_soc_Series}.
    The smallest non-negative integer $n$ such that $S=\Soc_n\left(S\right)$ is called \emph{left annihilator nilpotency index} of $S$.
\end{defin}

\medskip

\begin{prop}\label{prop_right_nilp}
  Let $(S,+,\circ)$ be a left annihilator nilpotent dual weak brace.  Then, $S$ is right nilpotent.\end{prop}
\begin{proof}
 Let $\E(S) = I_0\subseteq I_1\subseteq I_2\subseteq\cdots\subseteq I_m = S$ an $s$-series of $S$. Let $k\in \{1,\ldots,m\}$, we prove that $S^{(k)}\subseteq I_{m-k+1}$ by proceeding by induction on $k$. If $k=1$, then $S=I_m$. Now, let $k \in \mathbb{N}$ and assume that $S^{(k)}\subseteq I_{m-k+1}$. Then,  considered the canonical epimorphism $\chi: S \to S/I_{m-k}$, it holds that $\chi\left(I_{m-k+1}\right) \cdot \chi\left(S\right) \subseteq E\left(S/I_{m-k}\right)$, and so $I_{m-k+1} \cdot S \subseteq I_{m-k}$. Hence,
 \begin{align*}
     S^{(k+1)}=S^{(k)}\cdot S \subseteq I_{m-k+1} \cdot S \subseteq I_{m-k}
 \end{align*}
 and the claim follows by induction. Therefore, $S^{(m+1)}=E\left(S\right)$ and so $S$ is right nilpotent of right nilpotency index less or equal to $m$.
\end{proof}

\medskip

\noindent The converse of \cref{prop_right_nilp} is true in the particular case in which $(S,+)$ is a nilpotent Clifford semigroup. The notion of nilpotent Clifford semigroup which we adopt in this work is consistent with the one given in \cite{Me88}.

\begin{prop}
Let $(S,+, \circ)$ be a dual weak brace such that $(S,+)$ is nilpotent. Then, $S$ is right nilpotent if and only if $S$ is left annihilator nilpotent.\end{prop}
    \begin{proof}
    The necessary condition follows by \cref{prop_right_nilp}.
    The sufficient one is similar to the proof of \cite[Lemma 2.16]{CeSmVe19} by using lower central series of $\left(S, +\right)$ in  \cite[Definition 3.4]{Me88}.
    \end{proof}

\medskip

Below, we relate the left annihilator nilpotency of a dual weak brace $S=[Y; B_\alpha;\phii{\alpha}{\beta}]$ with those of each skew brace $B_\alpha$. It will be a direct consequence of the following lemma.
\begin{lemma}\label{lemma.sock}
    Let $S=[Y; B_\alpha;\phii{\alpha}{\beta}]$ be a dual weak brace. Then, 
$B_\alpha\cap \, \Soc_k\left(S\right)\subseteq \Soc_k\left(B_\alpha\right)$, for all $\alpha \in Y$ and $k \in \mathbb{N}$.\end{lemma}
    \begin{proof}
        Let us check that $B_\alpha\cap \, \Soc_k\left(S\right)\subseteq \Soc_k\left(B_\alpha\right)$, by proceeding by induction on $k$. If $k = 1$ the claim follows by  \cref{teo_descr_ideali}. Now, suppose that the statement holds for $k\in\mathbb{N}$ and let $a\in B_\alpha\cap\Soc_{k+1}\left(S\right)$. Then, for any $b\in B_\alpha$,
    $a\cdot b\in \Soc_{k}\left(S\right)$, $[a,b]_+\in\Soc_{k}\left(S\right)$, and $a\cdot b\in B\alpha$, and so, by the inductive hypothesis, $a\in\Soc_{k+1}\left(B_\alpha\right)$.
    \end{proof}

\begin{prop}\label{pro_right_balpha}
    Let $S=[Y; B_\alpha;\phii{\alpha}{\beta}]$ be a left annihilator nilpotent dual weak brace of index $n$. Then, for every $\alpha\in Y$ each skew brace $B_\alpha$ is left annihilator nilpotent of index less or equal to $n$. \\
    Vice versa, if each skew brace $B_\alpha$ is left annihilator nilpotent of index $k_\alpha$, then $S$ is left annihilator nilpotent of index equal to the maximum between $k_\alpha$, for every $\alpha \in Y$.
\end{prop}

\medskip

%*************************
\section{Annihilator nilpotency of dual weak braces}
%*************************

Let us start by introducing a further ideal of a dual weak brace $(S, +,\circ)$, that is the \emph{annihilator} of $S$. Denoted by $\zeta\left(S,\circ\right)$ the center of $\left(S,\circ\right)$, it is the set
    \begin{align*}
        \Ann\left(S\right) := \Soc\left(S\right)\cap \zeta\left(S,\circ\right),
    \end{align*}
namely,
$\Ann\left(S\right) = \{a \ | \ a\in S\quad  \forall\, b\in B\quad a + b = b + a = a\circ b = b\circ a\}$. Note that the definition is consistent with that given in \cite[Definition 7]{CCoSt19} for skew braces.
\begin{prop}
Let $(S, +, \circ)$ be a dual weak brace. Then, $\Ann\left(S\right)$ is an ideal of $S$.\end{prop}
\begin{proof}
  Clearly, $E\left(S\right)$ is contained in $\Ann\left(S\right)$.  
  Moreover, if $x\in \Ann\left(S\right)$ and $a\in S$, 
  \begin{align*}
      b\circ\left(-a + x + a\right)
      &= b - b\circ a + b^0 + b\circ x - b + b\circ a &\mbox{by \cref{prop_weak}-$2.$}\\
      &= b - b\circ a + x\circ b - b  + b\circ a 
      &\mbox{$x\in \zeta\left(S,\circ\right)$}\\
      &=  b - b\circ a + x + b^0  + b\circ a
      &\mbox{$x\in \Soc\left(S\right)$}\\
      &= b + x + \left(b\circ a\right)^0
      &\mbox{$x\in \Soc\left(S\right)$}\\
      &= b + x  + a^0+ b^0\\
      &= \left(- a + x + a \right)+ b &\mbox{$x\in \Soc\left(S\right)$}\\
      &= \left(- a + x + a \right)\circ b.
      &\mbox{$- a + x + a\in\Soc\left(S\right)$}
  \end{align*}
  for every $b\in S$. Besides, one has that $\lambda_a\left(x\right) = - a + x\circ a =-a+x+a\in\Ann\left(S\right)$.
  Finally, $a^-\circ x\circ a = a^0\circ x=\lambda_a(x)\in\Ann\left(S\right)$, which completes the proof.
\end{proof}

\medskip

As it happens for the socle, in general, the annihilator of a dual weak brace $S$ does not coincide with the union of the annihilators of each skew brace $B_\alpha$ (see, for instance, \cref{ex_C3_sym3}). Similarly to \cref{soc-ideale}, we obtain the following result.

\begin{prop}
  Let  $S=[Y; B_{\alpha}; \phii{\alpha}{\beta}]$ be a dual weak brace, $\psi_{\alpha,\beta} := {\phi_{\alpha,\beta}}_{|_{\Ann(B_\alpha)}}$, for all $\alpha \geq \beta$, and assume that $I:=[Y; \Ann\left(B_{\alpha}\right); \psi_{\alpha, \beta}]$ is an ideal of $S$. Then, $I=\Ann(S)$.
  \end{prop}
  
\medskip

\begin{rem}\label{rem_annS_puntino}
As a consequence of \cref{idemp.puntino}, if $(S, +, \circ)$ is a dual weak brace, it is also easy to see that  \begin{align*}
     \Ann\left(S \right)&=\{a \, | \, a \in S,\,\,\, \forall \ b \in S \quad a\cdot b=b \cdot a =[a,b]_{+}=a^0+b^0\}.
\end{align*}
\end{rem}

\medskip

Similarly to \cite{JeVAnVen22}, as is usual in ring theory,  we define the $k$-th annihilator of a dual weak brace $(S,+, \circ)$.
\begin{defin} Let $(S,+, \circ)$ be a dual weak brace. Set $\Ann_0\left(S\right):= E\left(S\right)$, we define $\Ann_k\left(S\right)$ to be the ideal of $S$ containing $\Ann_{k-1}\left(S\right)$ such that
    \begin{align*}
       \Ann_k\left(S\right)/\Ann_{k-1}\left(S\right)=\Ann\left(S/\Ann_{k-1}\left(S\right)\right),
    \end{align*}
    for every positive integer $k$.\end{defin}

\noindent Note that 
$\Ann_k\left(S\right)=\{a \, | \, a\in S, \, \,  \forall\, b \in S \quad a \cdot b, \,  b \cdot a, \, [a,b]_{+} \in \Ann_{k-1}\left(S\right)\},$ for every positive integer $k$.

\medskip

\begin{defin}
A dual weak brace $(S, +, \circ)$ is said to be \emph{annihilator nilpotent} if  there exists $n \in \mathbb{N}$ such that $\Ann_n\left(S\right) = S$.
\end{defin}
\smallskip
Clearly, any ideal $I$ and quotient $S/I$ of an annihilator nilpotent dual weak brace $S$ are annihilator nilpotent. The next example shows that, unlike what happens in ring theory, the converse is not true.

\begin{ex}
    Let $B:= \left(\mathbb{Z}_6,+,\circ\right)$ be the brace having as an additive group the cyclic group of $6$ elements and multiplication defined by $a\circ b := a + \left(-1\right)^{a}b$, for all $a,b\in B$. Then, $I:=\{0,1,2\}$ is an ideal of $B$ such that $\Ann\left(I\right) = I$, so $I$ is annihilator nilpotent. Clearly, $\Ann\left(B/I\right) = B/I$, so $B/I$ also is annihilator nilpotent. However, $B$ is not annihilator nilpotent since  $\Ann\left(B\right) = \{0\}$.
\end{ex}

\medskip

Following \cite[Definition 2.3]{BoJed21x}, given a dual weak brace $(S, +, \circ)$ and  an ideal $I$ of $S$, set $\Gamma_0\left(I\right):= I$, we can inductively define 
\begin{align*}
    \Gamma_k\left(I\right):= 
    \bigl\langle \,\Gamma_{k-1}\left(I\right) \cdot S, \  S \cdot \Gamma_{k-1}\left(I    \right), \ [\Gamma_{k-1}\left(I\right), S]_+ \,\bigr\rangle_+.
\end{align*}
 Observe that, set $\Gamma(I):=\Gamma_1(I)$, then $\Gamma_k(I)=\Gamma\left(\Gamma_{k-1}(I)\right)$, for every $k\in\mathbb{N}$.

\begin{prop}\label{prop_gamma}
    Let $(S, +, \circ)$ be a dual weak brace and $I$ an ideal of $S$. Then, $\Gamma_k(I)$ is an ideal of $S$, for every $k \in \mathbb{N}$.\end{prop}
    \begin{proof}
        Assuming $S=[Y; B_\alpha;\phii{\alpha}{\beta}]$, by an easy induction, one can show that $\phi_{\alpha, \beta}\left(\Gamma_k(B_\alpha)\right)\subseteq \Gamma_k\left(B_\beta\right)$, for any $\alpha > \beta$. Hence, the proof follows by \cref{teo_descr_ideali}.
    \end{proof}
    \smallskip
    
One can show that if $(B, +, \circ)$ is a skew brace and $I$ an ideal of $B$, then $\Gamma(I) \subseteq [I, B]$, where $[I, B]$ denotes the \emph{commutator} introduced in \cite{BoFaPo23}, which is the smallest ideal of $B$ containing $[I, B]_+, [I, B]_\circ$, and the set $\{i \circ b-(i+b) \ \mid \ i \in I, \ b \in B\}$. We can not establish whether $\Gamma(I) = [I, B]$. \\
We give the following lemmas, which is useful to prove the main result of this section and can be compared with
\cite[Proposition 15]{BaEsJiPe23x} regarding commutators in the context of skew braces.

\begin{lemma}\label{lemma_MN}
Let $M,N$ ideals of a dual weak brace $(S,+ , \circ)$ with $M\subseteq N$. The following statements are equivalent:
\begin{enumerate}
    \item[(i)] $N/M \subseteq \Ann(S/M)$,
    \item[(ii)] $\Gamma(N)\subseteq M$.
\end{enumerate}\end{lemma}
\begin{proof}
Initially, assume that $N/M \subseteq \Ann(S/M)$. Hence, $M+(n \circ s)=M+(n+s)$, for all $n \in N$ and $s \in S$. Then,
\begin{align*}
    n \cdot s=-n+\underbrace{n^0+\left(n\circ s\right)}_{M+(n \circ s)}-s=\underbrace{(-n+m+n)}_{M}+s-s \in M,
\end{align*}
for some $m \in M$. Similarly, $s \cdot n \in M$. Furthermore, since $M+(n+s)=(s+n)+M$,
\begin{align*}
    [n,s]_+=-n-s+\underbrace{n^0+n+s}_{M+(n+s)}=\underbrace{-n-s+s+n}_{M}+m \in M,
\end{align*}
for some $m \in M$. Therefore, $\Gamma(N)\subseteq M$.
The rest of the claim directly follows by \cref{rem_annS_puntino}.
\end{proof}

\medskip

\begin{lemma}\label{lemma_GAMMA_somma}
Let $(S, +, \circ)$ be a dual weak brace and $M, N$ ideals of $S$. Then, it holds that $\Gamma(M+N)=\Gamma(M)+\Gamma(N)$.\end{lemma}
\begin{proof}
It is a direct consequence of the definition of $\Gamma(I)$, for every ideal $I$ of $S$, and \cref{prop:cdot-circ}-\ref{ref_puntino2} and \ref{ref_puntino3}.
\end{proof}

\medskip 

\begin{defin}
Let $(S,+ , \circ)$ be a dual weak brace. An ascending series
\begin{align*}
    \E(S) = I_0\subseteq I_1\subseteq \cdots \subseteq  I_k=S
\end{align*}
of ideals of $S$ is called an \emph{annihilator series} of $S$ if $I_{j+1}/I_j\subseteq  \Ann\left(S/I_j\right)$, for $0 \leq j \leq k$. 
\end{defin}

\medskip

\begin{theor}\label{teo_nilp}
    Let $(S,+,\circ)$ be a dual weak brace. If $\E(S) = I_0\subseteq I_1\subseteq \cdots \subseteq I_k=S$ is an annihilator series of $S$, then
    \begin{equation*}
      \Gamma_{k-j}(S) \subseteq 
      I_j\subseteq \Ann_j(S),
    \end{equation*}
    for $0\leq j\leq k$. \end{theor}
\begin{proof} 
    First, we set $i:=k-j$ and show that $\Gamma_i(S)\subseteq I_{k-i}$, by proceeding by induction on $i$. If $i=0$, then $\Gamma_0(S)=S=I_k$. Now, let $i\in \mathbb{N}$ and assume that $\Gamma_i(S)\subseteq I_{k-i}$. Thus, by \cref{lemma_MN},
    \begin{equation*}
        \Gamma_{i+1}(S)=\Gamma(\Gamma_i(S))\subseteq \Gamma(I_{k-i})\subseteq I_{k-(i+1)},
    \end{equation*}
which completes the first part of the proof. \\
Now, we prove that $I_j \subseteq \Ann_j(S)$ by induction on $j$. If $j=0$, then $I_0=\E(S)=\Ann_0(S)$. Now, let $j \in \mathbb{N}$ and assume that $I_j \subseteq \Ann_j(S)$. By \cref{lemma_GAMMA_somma} and \cref{lemma_MN}, we have that
\begin{align*}
    \Gamma\left(I_j+\Ann_{j-1}(S)\right)=\Gamma(I_j)+\Gamma\left(\Ann_{j-1}(S)\right) \subseteq I_{j-1}+\Ann_{j-1}(S) \subseteq \Ann_{j-1}(S),
\end{align*}
where the last inclusion follows by the inductive hypothesis. Thus, since by \cref{lemma_MN}
\begin{align*}
   \left( I_j+\Ann_{j-1}(S)\right)/ \Ann_{j-1}(S) \subseteq \Ann\left(S/\Ann_{j-1}(S)\right)=\Ann_j(S)/\Ann_{j-1}(S),
\end{align*}
 we obtain that $I_j+ \Ann_{j-1}(S) \subseteq \Ann_{j}(S)$. Consequently, $I_j \subseteq \Ann_{j}(S)$ and the claim follows.
\end{proof}

\medskip In light of \cref{teo_nilp} we can give the following definitions.

\begin{defin}Let $(S, +, \circ)$ be a dual weak brace. The smallest non-negative integer $c$ such that $\Gamma_{c+1}(S)=\E(S)$ and $\Ann_c(S)=S$
is called \emph{nilpotency index} (or \emph{annihilator nilpotency index}) of $S$.
The series 
$$
       \E(S) = \Ann_0(S)\subseteq \Ann_1(S)\subseteq\cdots\subseteq \Ann_c(S) = S
$$ is called \emph{upper annihilator series} of $S$.  The series 
$$
       \E(S) = \Gamma_{c+1}(S) \subseteq \Gamma_c(S)\subseteq\cdots\subseteq \Gamma_0(S) = S
$$ 
is named \emph{lower annihilator series} of $S$.
\end{defin}

\medskip

Similar to \cref{lemma.sock} and \cref{pro_right_balpha}, one can prove the following result. 
\begin{prop}
    Let $S=[Y; B_\alpha;\phii{\alpha}{\beta}]$ be an annihilator nilpotent dual weak brace of nilpotency index $n$. Then, for every $\alpha\in Y$ each skew brace $B_\alpha$ is annihilator nilpotent of index less or equal to $n$.
     \\
Vice versa, if each skew brace $B_\alpha$ is annihilator nilpotent of index $k_\alpha$, then $S$ is annihilator nilpotent of index equal to the maximum between $k_\alpha$, for every $\alpha \in Y$.
 \end{prop}

\bigskip

%--------------------------------------
%\section*{References}
%-------------------------------------
\bibliographystyle{elsart-num-sort}  
\bibliography{bibliography}

\def\cprime{$'$}
\begin{thebibliography}{10}
\expandafter\ifx\csname url\endcsname\relax
  \def\url#1{\texttt{#1}}\fi
\expandafter\ifx\csname urlprefix\endcsname\relax\def\urlprefix{URL }\fi

\bibitem{BaEsJiPe23x}
A.~Ballester-Bolinches, R.~Esteban-Romero, P.~Jim\'{e}nez-Seral,
  V.~P\'{e}rez-Calabuig, On solubility of skew left braces and solutions of the
  {Y}ang-{B}axter equation, arxiv.org/abs/2304.13475.
\newline\urlprefix\url{https://arxiv.org/abs/2304.13475}

\bibitem{BaGu21}
V.~G. Bardakov, V.~Gubarev, Rota-{B}axter groups, skew left braces, and the
  {Y}ang-{B}axter equation, J. Algebra 596 (2022) 328--351.
\newline\urlprefix\url{https://doi.org/10.1016/j.jalgebra.2021.12.036}

\bibitem{Ba72}
R.~J. Baxter, Partition function of the eight-vertex lattice model, Ann.
  Physics 70 (1972) 193--228.
\newline\urlprefix\url{https://doi.org/10.1016/0003-4916(72)90335-1}

\bibitem{BoJed21x}
M.~Bonatto, P.~Jedlička, Central nilpotency of skew braces, J. Algebra Appl.,
  In press.
\newline\urlprefix\url{https://doi.org/10.1142/S0219498823502559}

\bibitem{BoFaPo23}
D.~Bourn, A.~Facchini, M.~Pompili, Aspects of the category {SKB} of skew
  braces, Comm. Algebra 51~(5) (2023) 2129--2143.
\newline\urlprefix\url{https://doi.org/10.1080/00927872.2022.2151609}

\bibitem{CaCeSt22}
F.~Catino, F.~Ced\'{o}, P.~Stefanelli, Nilpotency in left semi-braces, J.
  Algebra 604 (2022) 128--161.
\newline\urlprefix\url{https://doi.org/10.1016/j.jalgebra.2022.04.004}

\bibitem{CCoSt19}
F.~Catino, I.~Colazzo, P.~Stefanelli, Skew left braces with non-trivial
  annihilator, J. Algebra Appl. 18~(2) (2019) 1950033, 23.
\newline\urlprefix\url{https://doi.org/10.1142/S0219498819500336}

\bibitem{CCoSt20x-2}
F.~Catino, I.~Colazzo, P.~Stefanelli, Set-theoretic solutions to the
  {Y}ang--{B}axter equation and generalized semi-braces, Forum Math. 33~(3)
  (2021) 757--772.
\newline\urlprefix\url{https://doi.org/10.1515/forum-2020-0082}

\bibitem{CaMaMiSt22}
F.~Catino, M.~Mazzotta, M.~M. Miccoli, P.~Stefanelli, Set-theoretic solutions
  of the {Y}ang-{B}axter equation associated to weak braces, Semigroup Forum
  104~(2) (2022) 228--255.
\newline\urlprefix\url{https://doi.org/10.1007/s00233-022-10264-8}

\bibitem{CaMaSt23}
F.~Catino, M.~Mazzotta, P.~Stefanelli, Rota–{B}axter operators on {C}lifford
  semigroups and the {Y}ang–{B}axter equation, J. Algebra 622 (2023)
  587--613.
\newline\urlprefix\url{https://doi.org/10.1016/j.jalgebra.2023.02.013}

\bibitem{Ce18}
F.~Ced\'{o}, Left braces: solutions of the {Y}ang-{B}axter equation, Adv. Group
  Theory Appl. 5 (2018) 33--90.
\newline\urlprefix\url{https://doi.org/10.4399/97888255161422}

\bibitem{CeSmVe19}
F.~Ced\'{o}, A.~Smoktunowicz, L.~Vendramin, Skew left braces of nilpotent type,
  Proc. Lond. Math. Soc. (3) 118~(6) (2019) 1367--1392.
\newline\urlprefix\url{https://doi.org/10.1112/plms.12209}

\bibitem{Ci23AGTA}
R.~Civino, V.~Fedele, N.~Gavioli, Conference talk ``{B}i-braces and connections
  to private-key cryptography'', Advances in Group Theory and Applications
  2023, Lecce, 5-9 June 2023.
\newline\urlprefix\url{https://www.advgrouptheory.com/agta2023/program.html#7}

\bibitem{ClPr61}
A.~H. Clifford, G.~B. Preston, The algebraic theory of semigroups. {V}ol. {I},
  Mathematical Surveys, No. 7, American Mathematical Society, Providence, R.I.,
  1961.

\bibitem{Dr92}
V.~G. Drinfel\cprime~d, On some unsolved problems in quantum group theory, in:
  Quantum groups ({L}eningrad, 1990), vol. 1510 of Lecture Notes in Math.,
  Springer, Berlin, 1992, pp. 1--8.
\newline\urlprefix\url{https://doi.org/10.1007/BFb0101175}

\bibitem{ESS99}
P.~Etingof, T.~Schedler, A.~Soloviev, Set-theoretical solutions to the quantum
  {Y}ang-{B}axter equation, Duke Math. J. 100~(2) (1999) 169--209.
\newline\urlprefix\url{http://dx.doi.org/10.1215/S0012-7094-99-10007-X}

\bibitem{GuVe17}
L.~Guarnieri, L.~Vendramin, Skew braces and the {Y}ang-{B}axter equation, Math.
  Comp. 86~(307) (2017) 2519--2534.
\newline\urlprefix\url{https://doi.org/10.1090/mcom/3161}

\bibitem{Ho95}
J.~M. Howie, Fundamentals of semigroup theory, vol.~12 of London Mathematical
  Society Monographs. New Series, The Clarendon Press, Oxford University Press,
  New York, 1995, {O}xford Science Publications.

\bibitem{JeKuVaVe19}
E.~Jespers, {\L}.~Kubat, A.~Van~Antwerpen, L.~Vendramin, Factorizations of skew
  braces, Math. Ann. 375~(3-4) (2019) 1649--1663.
\newline\urlprefix\url{https://doi.org/10.1007/s00208-019-01909-1}

\bibitem{JeVAnVen22}
E.~Jespers, A.~Van~Antwerpen, L.~Vendramin, Nilpotency of skew braces and
  multipermutation solutions of the {Y}ang-{B}axter equation, Comm. Cont.
  Math., Article in press.
\newline\urlprefix\url{https://doi.org/10.1142/S021919972250064X}

\bibitem{KoTr20}
A.~Koch, P.~J. Truman, Opposite skew left braces and applications, J. Algebra
  546 (2020) 218--235.
\newline\urlprefix\url{https://doi.org/10.1016/j.jalgebra.2019.10.033}

\bibitem{KoSmVe21}
A.~Konovalov, A.~Smoktunowicz, L.~Vendramin, On skew braces and their ideals,
  Exp. Math. 30~(1) (2021) 95--104.
\newline\urlprefix\url{https://doi.org/10.1080/10586458.2018.1492476}

\bibitem{La98}
M.~V. Lawson, Inverse semigroups, World Scientific Publishing Co., Inc., River
  Edge, NJ, 1998, {T}he theory of partial symmetries.
\newline\urlprefix\url{https://doi.org/10.1142/9789812816689}

\bibitem{MaRySt23x}
M.~Mazzotta, B.~Rybo{\l}owicz, P.~Stefanelli, Deformed solutions of the
  {Y}ang-{B}axter equation coming from dual weak braces and unital
  near-trusses, arxiv.org/abs/2304.05235.
\newline\urlprefix\url{https://arxiv.org/abs/2304.05235}

\bibitem{Me88}
J.~D.~P. Meldrum, Group theoretic results in {C}lifford semigroups, Acta Sci.
  Math. (Szeged) 52~(1-2) (1988) 3--19.

\bibitem{Ze19}
K.~Nejabati~Zenouz, Skew braces and {H}opf-{G}alois structures of {H}eisenberg
  type, J. Algebra 524 (2019) 187--225.
\newline\urlprefix\url{https://doi.org/10.1016/j.jalgebra.2019.01.012}

\bibitem{Pe84}
M.~Petrich, Inverse semigroups, Pure and Applied Mathematics (New York), John
  Wiley \& Sons, Inc., New York, 1984, a Wiley-Interscience Publication.

\bibitem{PeRe99}
M.~Petrich, N.~R. Reilly, Completely regular semigroups, vol.~23 of Canadian
  Mathematical Society Series of Monographs and Advanced Texts, John Wiley \&
  Sons, Inc., New York, 1999, a Wiley-Interscience Publication.

\bibitem{PuSmZe22}
D.~Pulji\'{c}, A.~Smoktunowicz, K.~Nejabati~Zenouz, Some braces of cardinality
  {$p^4$} and related {H}opf-{G}alois extensions, New York J. Math. 28 (2022)
  494--522.
\newline\urlprefix\url{http://nyjm.albany.edu/j/2022/28-19.html}

\bibitem{RaYa24}
N.~Rathee, M.~K. Yadav, Cohomology, extensions and automorphisms of skew
  braces, J. Pure Appl. Algebra 228~(2) (2024) Paper No. 107462.
\newline\urlprefix\url{https://doi.org/10.1016/j.jpaa.2023.107462}

\bibitem{Ru07}
W.~Rump, Braces, radical rings, and the quantum {Y}ang-{B}axter equation, J.
  Algebra 307~(1) (2007) 153--170.
\newline\urlprefix\url{https://doi.org/10.1016/j.jalgebra.2006.03.040}

\bibitem{Sm18}
A.~Smoktunowicz, On {E}ngel groups, nilpotent groups, rings, braces and the
  {Y}ang-{B}axter equation, Trans. Amer. Math. Soc. 370~(9) (2018) 6535--6564.
\newline\urlprefix\url{https://doi.org/10.1090/tran/7179}

\bibitem{SmVe18}
A.~Smoktunowicz, L.~Vendramin, On skew braces (with an appendix by {N}. {B}yott
  and {L}. {V}endramin), J. Comb. Algebra 2~(1) (2018) 47--86.
\newline\urlprefix\url{https://doi.org/10.4171/JCA/2-1-3}

\bibitem{Ve19}
L.~Vendramin, Problems on skew left braces, Adv. Group Theory Appl. 7 (2019)
  15--37.
\newline\urlprefix\url{https://doi.org/10.32037/agta-2019-003}

\bibitem{Ya67}
C.~N. Yang, Some exact results for the many-body problem in one dimension with
  repulsive delta-function interaction, Phys. Rev. Lett. 19 (1967) 1312--1315.
\newline\urlprefix\url{https://doi.org/10.1103/PhysRevLett.19.1312}

\end{thebibliography}

 \end{document}